\documentclass[twoside,a4paper,reqno,12pt]{amsart}
\usepackage{amsfonts, amsbsy, amsmath, amsthm, amssymb, latexsym, verbatim, enumerate}
\usepackage{mathrsfs}
\usepackage[top=30mm,right=30mm,bottom=30mm,left=30mm]{geometry}

\usepackage{hyperref}
\usepackage{amssymb,amsfonts,amsmath,amsopn,amstext,amscd,latexsym,amsthm}
\usepackage{bbm}
\usepackage[all,cmtip]{xy}

\usepackage{enumerate}
\usepackage[shortlabels]{enumitem}

\theoremstyle{plain}

\begin{document}

\def\a{\alpha}
 \def\b{\beta}
 \def\e{\epsilon}
 \def\d{\delta}
  \def\D{\Delta}
 \def\c{\chi}
 \def\k{\kappa}
 \def\g{\gamma}
 \def\Ind{\mathrm{Ind}}
 \def\t{\tau}
\def\ti{\tilde}
 \def\N{\mathbb N}
 \def\Q{\mathbb Q}
 \def\Z{\mathbb Z}
 \def\C{\mathbb C}
 \def\F{\mathbb F}
 \def\ovF{\overline\F}
 \def\bfN{\mathbf N}
 \def\cG{\mathcal G}
 \def\cT{\mathcal T}
 \def\cX{\mathcal X}
 \def\cY{\mathcal Y}
 \def\cC{\mathcal C}
 \def\cH{\mathcal H}
 \def\cZ{\mathcal Z}
 \def\cO{\mathcal O}
 \def\cW{\mathcal W}
 \def\cL{\mathcal L}
 \def\bfC{\mathbf C}
 \def\bfZ{\mathbf Z}
 \def\bfO{\mathbf O}
 \def\G{\Gamma}
 \def\BS{\bar{S}}
 \def\bO{\boldsymbol{\Omega}}
 \def\bgo{\boldsymbol{\omega}}
 \def\go{\rightarrow}
 \def\do{\downarrow}
 \def\ra{\rangle}
 \def\la{\langle}
 \def\fix{{\rm fix}}
 \def\ind{{\rm ind}}
 \def\rfix{{\rm rfix}}
 \def\diam{{\rm diam}}
 \def\uni{{\rm uni}}
 \def\diag{{\rm diag}}
 \def\Irr{{\rm Irr}}
 \def\Syl{{\rm Syl}}
 \def\Out{{\rm Out}}
 \def\Tr{{\rm Tr}}
 \def\M{{\cal M}}
 \def\CL{{\mathcal C}}
\def\td{\tilde\delta}
\def\tx{\tilde\xi}
\def\DC{D^\circ}
\def\ext{{\rm Ext}}
\def\res{{\rm Res}}
\def\Ker{{\rm Ker}}
\def\hom{{\rm Hom}}
\def\End{{\rm End}}
 \def\rank{{\rm rank}}
 \def\soc{{\rm soc}}
 \def\Cl{{\rm Cl}}
 \def\A{{\sf A}}
 \def\sP{{\sf P}}
 \def\sQ{{\sf Q}}
 \def\SSS{{\sf S}}
  \def\SQ{{\SSS^2}}
 \def\St{{\sf {St}}}
 \def\p{\ell}
 \def\ps{\ell^*}
 \def\SC{{\rm sc}}
 \def\supp{{\sf{supp}}}
  \def\cR{{\mathcal R}}
\newcommand{\tw}[1]{{}^#1}
\newcommand{\dl}{\mathfrak{d}}

\def\tr{{\rm tr}}
 \def\Sym{{\rm Sym}}
 \def\PSL{{\rm PSL}}
 \def\SL{{\rm SL}}
 \def\Sp{{\rm Sp}}
 \def\GL{{\rm GL}}
 \def\SU{{\rm SU}}
 \def\U{{\rm U}}
  \def\L{{\rm L}}
 \def\SO{{\rm SO}}
 \def\PO{{\rm P}\Omega}
 \def\Spin{{\rm Spin}}
 \def\PSp{{\rm PSp}}
 \def\PSU{{\rm PSU}}
 \def\PGL{{\rm PGL}}
 \def\PGU{{\rm PGU}}
 \def\Iso{{\rm Iso}}
 \def\Stab{{\rm Stab}}
 \def\GO{{\rm GO}}
 \def\Ext{{\rm Ext}}
 \def\E{{\cal E}}
 \def\l{\lambda}
 \def\ve{\varepsilon}
 \def\Lie{\rm Lie}
 \def\s{\sigma}
 \def\O{\Omega}
 \def\o{\omega}
 \def\ot{\otimes}
 \def\op{\oplus}
 \def\oc{\overline{\chi}}
\def\bZ{\mathbf{Z}}
\def\da{\downarrow}
 \def\pf{\noindent {\bf Proof.$\;$ }}
 \def\Proof{{\it Proof. }$\;\;$}
 \def\no{\noindent}
\def\skipa{\vspace{-1.5mm} & \vspace{-1.5mm} & \vspace{-1.5mm}\\}
\newcommand{\ta}{\hspace{0.5mm}^{2}\hspace*{-0.2mm}} 
\def\hal{\unskip\nobreak\hfil\penalty50\hskip10pt\hbox{}\nobreak
 \hfill\vrule height 5pt width 6pt depth 1pt\par\vskip 2mm}

 \renewcommand{\thefootnote}{}

\newtheorem{theorem}{Theorem}
 \newtheorem{thm}{Theorem}
 \newtheorem{prop}[thm]{Proposition}
 \newtheorem{conj}[thm]{Conjecture}
 \newtheorem{question}[thm]{Question}
 \newtheorem{lem}[thm]{Lemma}
 \newtheorem{lemma}[thm]{Lemma}
 \newtheorem{defn}[thm]{Definition}
 \newtheorem{cor}[thm]{Corollary}
 \newtheorem{coroll}[theorem]{Corollary}
\newtheorem*{corB}{Corollary}
 \newtheorem{rem}[thm]{Remark}
 \newtheorem{exa}[thm]{Example}
 \newtheorem{cla}[thm]{Claim}

\parskip 1mm

\title{Unitary $t$-Groups}

\author[E. Bannai]{Eiichi Bannai}
\address{E. Bannai, Professor Emeritus, Kyushu University, Fukuoka 819-0395, Japan}
\email{bannai@math.kyushu-u.ac.jp}

\author[G. Navarro]{Gabriel Navarro}
\address{G. Navarro, Departament of Mathematics, Universitat de Val\`encia,
 Dr. Moliner 50, 46100 Burjassot, Spain.}
\email{gabriel.navarro@uv.es} 
 
\author[N. Rizo]{Noelia Rizo}
\address{N. Rizo, Departament of Mathematics, Universitat de Val\`encia,
 Dr. Moliner 50, 46100 Burjassot, Spain.}
\email{noelia.rizo@uv.es} 
 
\author[P.H. Tiep]{Pham Huu Tiep}
\address{P. H. Tiep, Department of Mathematics, Rutgers University, Piscataway, NJ 08854, USA}
\email{tiep@math.rutgers.edu}
 
\date{\today}

\subjclass[2010]{ 05B30, 20C15, 81P45}

\thanks{The research of the second and third authors is partially supported by the
 Spanish Ministerio de Educaci\'on y Ciencia proyecto MTM2016-76196-P and Prometeo
 Generalitat Valenciana.  
The fourth author gratefully acknowledges the support of the NSF (grant DMS-1840702).
The paper is partially based upon work supported by the NSF under grant DMS-1440140 while 
the second, third, and fourth authors were in residence at MSRI (Berkeley, CA), during the Spring 2018
semester. It is a pleasure to thank the Institute for the hospitality and support.}

\begin{abstract}
Relying on the main results of \cite{GT}, we classify all unitary $t$-groups for $t \geq 2$ in any dimension 
$d \geq 2$. We also show that there is essentially a unique unitary $4$-group, which is also a unitary $5$-group, but 
not a unitary $t$-group for any $t \geq 6$.
\end{abstract}

\maketitle


Unitary $t$-designs  have recently attracted a lot of interest in quantum information theory.
The concept of unitary $t$-design was first conceived in physics 
community as a finite set that approximates the unitary group 
$\U_d(\C)$, like any other design concept. It seems that works of 
Gross--Audenaert--Eisert \cite{GAE} and Scott \cite{Sc} marked the start 
of the research on unitary $t$-designs. Roy--Scott \cite{RS} gives a 
comprehensive study of unitary $t$-designs from a mathematical 
viewpoint. 

It is known that unitary $t$-designs in $\U_d(\C)$ always exist for 
any $t$ and $d$, but explicit constructions are not so easy in general. 
A special interesting case is the case where a unitary $t$-design 
itself forms a {\it group}. Such a finite group in $\U_d(\C)$ is called a 
{\it unitary $t$-group}. Some examples of unitary 5-groups are known 
in $\U_2(\C)$. For $d\geq 3$, some unitary 3-groups have been known 
in $\U_d(\C)$. But no example of unitary 4-groups in dimensions $d \geq 3$ was known. 
It seems that the difficulty of finding 4-groups in $\U_d(\C)$ for 
$d\geq 3$ has been noticed by many researchers (see e.g. Section 1.2 
of \cite{ZKGG}). The purpose of this paper is to clarify this situation. 
Namely, we point out that this problem in dimensions $\geq 4$ is essentially solved 
in the context of finite group theory by Guralnick--Tiep \cite{GT}. We also 
show that the classification of unitary $2$-groups in $\U_d(\C)$ for $d\geq 5$ is 
derived from \cite{GT} as well. Building on this, we provide a complete description of 
unitary $t$-groups in $\U_d(\C)$ for all $t,d \geq 2$.

We now recall the notion of unitary $t$-groups, following \cite[Corollary 8]{RS}.
Let $V = \C^d$ be endowed with standard Hermitian form and let $\cH = \U(V) = \U_d(\C)$ denote the corresponding unitary group.
Then a finite subgroup $G < \cH$ is called a 
{\it unitary $t$-group}
for some integer $t \geq 1$, if 
\begin{equation}\label{unit1}
  \frac{1}{|G|}\sum_{g \in G}|\tr(g)|^{2t} = \int_{X \in \cH}|\tr(X)|^{2t}dX.
\end{equation}
Note that the right-hand-side in \eqref{unit1} is exactly the {\it $2t$-moment} $M_{2t}(\cH,V)$ as defined in \cite{GT},
whereas the left-hand-side is the $2t$-moment $M_{2t}(G,V)$.  Recall, see e.g. \cite[\S26.1]{FH}, 
that the complex irreducible representations of the real Lie algebra $\mathfrak{su}_d$ 
and the complex Lie algebra $\mathfrak{sl}_d$ are the same. It follows that $M_{2t}(\cH,V) = M_{2t}(\cG,V)$ for 
$\cG = \GL(V)$.
Given these basic observations, we can recast the main results of \cite{GT} in the finite setting as follows. 

\begin{thm}\label{gt-main1}
Let $V = \C^{d}$ with $d \geq 5$ and $\cG = \GL(V)$. Assume that $G < \cG$ is a 
finite subgroup. Then $M_{8}(G,V) > M_{8}(\cG,V)$. In particular, if $d \geq 5$ and $t \geq 4$, then there does not exist any 
unitary $t$-group in $\U_d(\C)$.
\end{thm}

\begin{proof}
The first statement is precisely \cite[Theorem 1.4]{GT}. The second statement then follows from the first and 
\cite[Lemma 3.1]{GT}.
\end{proof}

We note that \cite[Theorem 1.4]{GT} also considers any Zariski closed subgroups $G$ of $\cG$ with the connected 
component $G^\circ$ being reductive. Then the only extra possibility with $M_8(G,V) = M_8(\cG,V)$ is when $G \geq [\cG,\cG] = \SL(V)$.
In fact, \cite{GT} also considers the problem in the modular setting.

Combined with Theorem \ref{main23} (below), Theorem \ref{gt-main1} yields the following consequence, which gives the complete classification
of unitary $t$-groups for any $t \geq 4$:

\begin{cor}\label{main-u4g}
Let $G < \U_d(\C)$ be a finite group and $d \geq 2$. Then $G$ is a unitary $t$-group for some $t \geq 4$ if and only if 
$d = 2$, $t = 4$ or $5$, and $G = \bZ(G)\SL_2(5)$.
\end{cor}

Next, we obtain the following consequences of \cite[Theorems 1.5, 1.6]{GT}, where $F^*(G)=F(G)E(G)$ denotes the generalized Fitting subgroup of any finite group $G$ (respectively, $F(G)$ is the Fitting subgroup and $E(G)$ is the layer of $G$); furthermore, we follow the notation of \cite{Atlas} for various simple groups. We also refer the reader to \cite{GMST} and 
\cite{TZ2} for the definition and basic properties of {\it Weil representations} of (certain) finite classical groups.

\begin{thm}\label{gt-complex}
Let $V = \C^{d}$ with $d \geq 5$ and let $\cG = \GL(V)$. For any finite subgroup $G < \cG$,
set $\BS = S/\bZ(S)$ for $S := F^{*}(G)$. Then $M_{4}(G,V) = M_{4}(\cG,V)$ if and only if one of the following conditions holds.

\begin{enumerate}[\rm(i)]
\item {\sf (Lie-type case)} One of the following holds.

\begin{enumerate}[\rm(a)]
\item $\BS = \PSp_{2n}(3)$, $n \geq 2$, $G = S$, and
$V \da_{S}$ is a Weil module of dimension $(3^n \pm 1)/2$.

\item $\BS = \U_{n}(2)$, $n \geq 4$, $[G:S] = 1$ or $3$, and $V \da_{S}$ is a
Weil module of dimension $(2^{n}-(-1)^{n})/3$.
\end{enumerate}

\item {\sf (Extraspecial case)} $d = p^{a}$ for some prime $p$ and $F^{*}(G) = F(G)=\bZ(G)E$, where $E = p^{1+2a}_+$ is an extraspecial  
$p$-group of order $p^{1+2a}$ and type $+$. 
Furthermore, $G/\bZ(G)E$ is a subgroup of $\Sp(W) \cong \Sp_{2a}(p)$ that acts transitively on
$W \smallsetminus \{0\}$ for $W = E/\bZ(E)$, and so is listed in Theorem \ref{trans} (below). If $p > 2$ then $E \lhd G$; if $p = 2$ then 
$F^*(G)$ contains a normal subgroup $E_1 \lhd G$, where $E_1 = C_4 * E$ is a central product of order $2^{2a+2}$ of 
$\bZ(E_1)=C_4 \leq \bZ(G)$ with $E$.

\item {\sf (Exceptional cases)} $S=\bZ(G)[G^*,G^*]$, and $(\dim(V), \BS, G^*)$ is as listed in Table {\rm I}. Furthermore,
in all but lines {\rm 2--6} of Table {\rm I}, $G = \bZ(G)G^*$. In lines {\rm 2--6}, either $G = S$ or $[G:S] = 2$ and $G$ induces on $\BS$ the outer 
automorphism listed in the fourth column of the table.
\end{enumerate}
In particular, $G < \cH = \U(V)$ is a unitary $2$-group if and only if $G$ is as described in {\rm (i)--(iii)}.
\end{thm}

\vskip10pt
\begin{figure}[ht]
\centerline
{{\sc Table} I. Exceptional examples in $\cG = \GL_d(\C)$ with $d \geq 5$}
\begin{center}
\begin{tabular}{|r||c|c|c|c|c|} \hline \skipa
     $d$ & $\BS$ & $G^*$ & Outer & 
     {\small \begin{tabular}{c}The largest $2k$ with\\ $M_{2k}(G,V) = M_{2k}(\cG,V)$ \end{tabular}}& 
         \begin{tabular}{c}$M_{2k+2}(G,V)$ vs.\\$M_{2k+2}(\cG,V)$ \end{tabular}\\ 
          \skipa \hline \hline \skipa
     $6$ & $\A_{7}$ & $6\A_{7}$ &  & $4$ & $21$ vs. $6$\\
     $6$ & $\L_{3}(4)$$^{\;(\star)}$ & $6\L_{3}(4) \cdot 2_{1}$ & $2_1$ & $6$ & $56$ vs. $24$\\
     $6$ & $\U_{4}(3)$$^{\;(\star)}$ & $6_{1} \cdot \U_{4}(3)$ & $2_2$ & $6$ & $25$ vs. $24$\\
\hline
     $8$ & $\L_{3}(4)$ & $4_{1} \cdot \L_{3}(4)$ & $2_3$ & $4$ & $17$ vs. $6$\\
\hline
     $10$ & $M_{12}$ & $2M_{12}$ & $2$ & $4$ & $15$ vs. $6$\\ 
     $10$ & $M_{22}$ & $2M_{22}$ & $2$ & $4$ & $7$ vs. $6$\\ \hline
     $12$ & $Suz$$^{\;(\star)}$ & $6Suz$ &  & $6$ & $25$ vs. $24$\\ \hline
     $14$ & $\ta B_{2}(8)$ & $\ta B_{2}(8) \cdot 3$ & & $4$ & $90$ vs. $6$\\
     $18$ & $J_{3}$$^{\;(\star)}$ & $3J_{3}$ &  & $6$ & $238$ vs. $24$\\ \hline
     $26$ & $\ta F_{4}(2)'$ & $\ta F_{4}(2)'$ &  & $4$ & $26$ vs. $6$\\ \hline
     $28$ & $Ru$ & $2Ru$ &  & $4$ & $7$ vs. $6$\\ \hline  
     $45$ & $M_{23}$ & $M_{23}$ &  & $4$ & $817$ vs. $6$\\ 
     $45$ & $M_{24}$ & $M_{24}$ &  & $4$ & $42$ vs. $6$\\ \hline
     $342$ & $O'N$ & $3O'N$ &  & $4$ & $3480$ vs. $6$\\ \hline
     $1333$ & $J_{4}$ & $J_{4}$ &  & $4$ & $8$ vs. $6$\\ \hline
\end{tabular}
\end{center}
\end{figure}

Note that in Table I, the data in the sixth column is given when we take $G = G^*$.

\begin{proof}
We apply \cite[Theorem 1.5]{GT} to $(G,\cG)$. Then case (A) of the theorem is impossible as $G$ is finite, and case (D) leads to case 
(iii) as $\cG = \GL(V)$. 

In case (B) of 
\cite[Theorem 1.5]{GT}, we have that $\BS = \PSp_{2n}(q)$ with $n \geq 2$ and $q = 3,5$, or 
$\BS = \PSU_n(2)$ with $n \geq 4$, and $V\da_S$ is irreducible. It is easy to see that the latter condition implies that $G/S$ has order $1$ or $3$. Next, $L = E(G)$ is a quotient of $\Sp_{2n}(q)$ or $\SU_n(2)$ by a central subgroup, and $S = \bZ(S)L$.
Let $\chi$ denote the character of the $G$-module $V$. 
As $d > 4$, the condition $M_4(G,V) = M_4(\cG,V)$ is equivalent to that $G$ act irreducibly on both $\Sym^2(V)$ and $\wedge^2(\chi)$
(see the discussion in \cite[\S2]{GT}).
Hence, if $\chi\da_L$ is real-valued, then either $\Sym^2(\chi\da_L)$ or $\wedge^2(\chi\da_L)$ contains $1_L$, whence 
either $\Sym^2(\chi\da_S)$ or $\wedge^2(\chi\da_S)$ contains a linear character. But both $\Sym^2(V)$ and $\wedge^2(V)$ have dimension
at least $d(d-1)/2 \geq 10$ and $[G:S] \leq 3$, so $G$ cannot act irreducibly on them, a contradiction. We have shown that 
$\chi\da_L$ is not real-valued. Now using Theorems 4.1 and 5.2 of \cite{TZ1}, we can rule out the case $\BS = \PSp_{2n}(5)$ and 
the case $(\BS,\dim(V)) = (\PSU_n(2),(2^n+2(-1)^n)/3)$, as $\chi\da_L$ is real-valued in those cases. 

Case (C), together with \cite[Lemma 5.1]{GT}, leads to case (ii) listed above, except for the explicit description of $E$ and $E_1$. Suppose
$p > 2$. Then at least one element in $E \smallsetminus \bZ(E)$ has order $p$, whence all elements in $E \smallsetminus \bZ(E)$ 
have order $p$ by the transitivity of $G/\bZ(G)E$ on $W \smallsetminus \{0\}$, i.e. $E$ has type $+$. Also, note that $E$ is generated by all
elements of order $p$ in $\bZ(G)E$, and so $E \lhd G$. Next suppose that $p = 2$ and let $E_1 \lhd G$ be generated by all elements of order 
at most $4$ in $\bZ(G)E$. If $|\bZ(G)| < 4$, then $F^*(G) = E_1 = E$ is an extraspecial $2$-group of order $2^{1+2a}$ of type $\epsilon$
for some $\epsilon = \pm$. In this case, $G/\bZ(G)E \hookrightarrow O^{\epsilon}_{2a}(2)$ and so cannot be transitive on 
$W \smallsetminus \{0\}$ (as $a \geq 2$), a contradiction. So $|\bZ(G)| \geq 4$. In this case, one can show that $E_1 = C_4 * E$
with $\bZ(E) < C_4 \leq \bZ(G)$, and since $C_4 * 2^{1+2a}_+ \cong C_4 * 2^{1+2a}_-$, we may choose $E$ to have type $+$.
\end{proof} 
 
We note that the case of Theorem \ref{gt-complex} where $G$ is almost quasisimple was also treated in 
\cite{M}. More generally, the classification of subgroups of a classical group ${\mathsf {Cl}}(V)$ in characteristic $p$
that act irreducibly on the heart of the tensor square, symmetric square, or alternating square of 
$V \otimes_{\F_p}\overline{\F}_p$, is of particular importance to the Aschbacher-Scott program \cite{A} of 
classifying maximal groups of finite classical groups. See \cite{Mag}, \cite{MM}, \cite{MMT} for results on this problem
in the modular case. 
 
\begin{thm}\label{gt-complex-m6}
Let $V = \C^{d}$ with $d \geq 5$ and let $\cG = \GL(V)$. Assume $G$ is a
finite subgroup of $\cG$. Then $M_{6}(G,V) = M_{6}(\cG,V)$ if and only if one of the following 
two conditions holds.

\begin{enumerate}[\rm(i)]
\item {\sf (Extraspecial case)} $d = 2^{a}$ for some $a > 2$, and  
$G = \bZ(G)E_1 \cdot Sp_{2a}(2)$, where $E \cong 2^{1+2a}_{+}$ is extraspecial and of type $+$ and 
$E_1 = C_4*E$ with $C_4 \leq \bZ(G)$. 

\item {\sf (Exceptional cases)} Let $\BS = S/\bZ(S)$ for $S = F^*(G)$. Then 
$$\BS \in \{\L_{3}(4),\U_{4}(3),Suz,J_{3}\},$$ 
and $(\dim(V), \BS, G^*)$ is as listed in the lines marked by $^{(\star)}$ in Table {\rm I}. Furthermore, either 
$G = \bZ(G)G^*$, or 
$\BS = \U_4(3)$ and $S = \bZ(G)G^*$.
\end{enumerate}
In particular, $G < \cH = \U(V)$ is a unitary $3$-group if and only if $G$ is as described in {\rm (i), (ii)}.
\end{thm}

\begin{proof}
Apply \cite[Theorem 1.6]{GT} and also Theorem \ref{gt-complex}(ii) to $(G,\cG)$.
\end{proof}

The transitive subgroups of $\GL_n(p)$ are determined by Hering's theorem \cite{He} (see also \cite[Appendix 1]{L}), which however 
is not easy to use in the solvable case. For the complete determination of unitary 2-groups in Theorem \ref{gt-complex}(ii), 
we give a complete classification of such groups in the symplectic case that is needed for us. The notations such as
${\tt SmallGroup}(48,28)$ are taken from the {\tt SmallGroups} library in \cite{GAP}.
  
\begin{thm}\label{trans}
Let $p$ be a prime and let $W = \F_p^{2n}$ be endowed with a non-degenerate symplectic form.
Assume that a subgroup $H \le \Sp(W)$ acts transitively on $W \smallsetminus \{0\}$.
Then $(H,p,2n)$ is as in one of the following cases.
 
\begin{enumerate}[\rm(A)]
\item {\sf (Infinite classes)}:

\begin{enumerate}[\rm(i)]

\item $n = bs$ for some integers $b,s \geq 1$, and $\Sp_{2b}(p^s)' \lhd H \leq \Sp_{2b}(p^s) \rtimes C_s$.

\item $p=2$, $n = 3s$ for some integer $s \geq 2$; and $G_2(2^s) \lhd H \leq G_2(2^s) \rtimes C_s$.

\end{enumerate}

\item {\sf (Small cases)}:

\begin{enumerate}[\rm(i)]
\item
$(2n,p)=(2,3)$, and $H=Q_8$.
\item
$(2n,p)=(2,5)$, and $H=\SL_2(3)$. 

\item
$(2n,p)=(2,7)$, and $H=\SL_2(3).C_2={\tt SmallGroup}(48,28)$.

\item
$(2n,p)=(2,11)$, and $H=\SL_2(5)$.

\item $(2n,p)=(4,3)$, and $H={\tt SmallGroup}(160,199)$, ${\tt SmallGroup}(320,1581)$,
$2.{\sf S}_5$, $\SL_2(9)$, $\SL_2(9)\rtimes C_2={\tt SmallGroup}(1440,4591)$, or\\
$C_2 . ((C_2 \times C_2 \times C_2 \times C_2) \rtimes {\sf A}_5)={\tt SmallGroup}(1920,241003)$.

\item $(2n,p)=(6,2)$, and $H=\SL_2(8)$, $\SL_2(8) \rtimes C_3$, 
$\SU_3(3)$, $\SU_3(3) \rtimes C_2$.

\item $(2n,p) = (6,3)$ and $H = \SL_2(13)$.
\end{enumerate}
\end{enumerate}
 \end{thm}
 
 \begin{proof}
We may assume that $(2n,p)$ is not in one of the small cases listed in (B), which are computed using \cite{GAP}.
We have that $[H:\bfC_H(v)]=p^{2n}-1$, for every $v \in W \smallsetminus \{0\}$. Now we apply Hering's theorem,
as given in \cite[Appendix 1]{L} and analyze possible classes for $H$.

\smallskip
(a) Suppose that $H \le \Gamma {\mathrm L}_1(p^{2n})$, which is the semidirect product
 of $\Gamma_0$ (the multiplicative field of $\F_{p^{2n}}$) and the Galois
 automorphism $\sigma$ of order $2n$.
 If $n=1$, then $H \le \SL_2(p)$, which has order $p(p-1)(p+1)$, and we may assume that $p \geq 13$. As the smallest index 
 of proper subgroups of $\SL_2(p)$ is $p+1$ (see e.g. \cite[Table VI]{TZ1}), we conclude that $H = \SL_2(p)$.
 So we may assume that $n>1$. We may also assume that $(2n,p) \ne (2,6)$.
 Hence, we can consider a Zsigmondy (odd) prime divisor $r$ of  $p^{2n}-1$ \cite{Zs},
and have that the order of $p$ mod $r$ is $2n$. Thus $2n$ divides $r-1$.
 Let $C = H \cap \Gamma_0$.
  Note that $r$
divides $|C|$ (because $r$ does not divide $2n$),
and hence $C$ acts irreducibly on $W$.
 Since $C < \Sp(W)$, by \cite[Satz II.9.23]{Hu} we have
that
$|C|$ divides $p^n +1$. Hence, $|H|$ divides $2n(p^n+1)$,
and thus $p^n-1$ divides $2n$.  This is not possible.



\smallskip
(b) Aside from the possibilities listed in (A) and (B), we need only consider the possibility $2n = as$ with $a \geq 3$,
$p^n \neq 2^2$, $3^2$, $2^3$, $3^3$, and $H \rhd \SL_a(p^s)$. Let $\dl(X)$ denote the smallest degree of faifthful complex 
representations of a finite group $X$. Since $H \leq \Sp_{2n}(p)$,  by \cite[Theorem 5.2]{TZ1} we have that 
$$\dl(X) \leq (p^n+1)/2 = (p^{as/2}+1)/2.$$
On the other hand, since $H \rhd \SL_a(p^s)$, by \cite[Theorem  3.1]{TZ1} we also have that 
$$\dl(X) \geq (p^{as}-p^s)/(p^s-1) > p^{s(a-1)}.$$
As $a \geq 3$, this is impossible.
\end{proof}
  
Next we complete the classification of unitary $t$-groups in dimension $4$. First we introduce some key groups for this classification, where 
we use the notation of \cite{GAP} for ${\tt SmallGroup}(64,266)$ and ${\tt PerfectGroup}(23040,2)$.

\begin{prop}\label{norm}
Consider an irreducible subgroup 
$E_4 = C_4 *2^{1+4}_+ = {\tt SmallGroup}(64,266)$
of order $2^6$ of $\GL(V)$, where $V = \C^4$, and let 
$\Gamma_4:= \bfN_{\GL(V)}(E_4)$. Then the following statements hold.
\begin{enumerate}[\rm(i)]
\item $\Gamma_4$ induces the subgroup $A^+ \cong C_2^4 \cdot \SSS_6$ of all automorphisms of $E_4$ that act trivially on
$\bZ(E_4) = C_4$.
\item The last term $\Gamma_4^{(\infty)}$ of the derived series of $\Gamma_4$ is $L = {\tt PerfectGroup}(23040,2)$, a perfect group of order $23040$ and of shape $E_4 \cdot \A_6$. Furthermore, $\Gamma_4^{(\infty)}$ is a unitary $3$-group.
\end{enumerate}
\end{prop}
 
\begin{proof}
(i) It is well known, see e.g. \cite[p. 404]{Gr}, that $A^+ \cong {\mathrm {Inn}}(E_4) \cdot \SSS_6$ with ${\mathrm {Inn}}(E_4) \cong C_2^4$. 
Certainly, $\Gamma_4/\bfC_{\Gamma_4}(E_4) \hookrightarrow A^+$. Let 
$\psi$ denote the character of $E_4$ afforded by $V$, and note that $\psi$ and $\overline\psi$ are the only two irreducible 
characters of degree $4$ of $E_4$, and they differ by their restrictions to $\bZ(E_4)$. Now for any $\alpha \in A^+$,
$\psi^\alpha = \psi$. It follows that there is some $g \in \GL(V)$ such that $gxg^{-1} = \alpha(x)$ for all $x \in E_4$; in particular,
$g \in \Gamma_4$. We have therefore shown that $\Gamma_4/\bfC_{\Gamma_4}(E_4) \cong A^+$.

\smallskip
(ii) Using \cite{GAP}, one can check that $L := {\tt PerfectGroup}(23040,2)$ embeds in $\GL(V)$, with a character say $\chi$, and 
$F^*(L) \cong E_4$. So without loss we may identify $F^*(L)$ with $E_4$ and obtain that $L < \Gamma_4$. Again using \cite{GAP} we 
can check that $[\chi^3,\chi^3]_L = 6 = M_6(\GL(V))$, which means that $L$ is a unitary $3$-group. As $L$ is perfect, we have that 
$L \leq \Gamma_4^{(\infty)}$. Next, $L$ acting on $E_4$ induces the perfect subgroup $A^{++} \cong C_2^4 \cdot \A_6$ of index $2$ in $A^+$,
and the same also holds for $\Gamma_4^{(\infty)}$. Hence, for any $g \in \Gamma_4^{(\infty)}$, we can find $h \in L$ such that 
the conjugations by $g$ and by $h$ induce the same automorphism of $E_4$. By Schur's Lemma, $gh^{-1} \in \bZ(\Gamma_4)$, whence
$\Gamma_4^{(\infty)} \leq \bZ(\Gamma_4)L$. Taking the derived subgroup, we see that $\Gamma_4^{(\infty)} \leq L$, and so $\Gamma_4^{(\infty)} = L$,
as stated.
\end{proof}

Next, we recall three {\it complex reflection groups} $G_{29}$, $G_{31}$, and $G_{32}$ in dimension $4$, namely, the ones listed on lines
29, 31, and 32 of \cite[Table VII]{ST}. A direct calculation using the computer packages {{\sf {GAP}}3} \cite{Mi}, \cite{S+}, and {\sf {Chevie}}
\cite{GHMP}, shows that each of these $3$ groups $G$, being embedded in $\cH = \U_4(\C)$, is a unitary $2$-group. Also,
$$F(G_{29}) \cong F(G_{31}) \cong {\tt SmallGroup}(64,266),~F(G_{32}) = \bZ(G_{32}) \cong C_6,$$
and 
$$G_{29}/F(G_{29}) \cong \SSS_5,~G_{31}/F(G_{31}) \cong \SSS_6,~G_{32} \cong C_3 \times \Sp_4(3).$$ 
In what follows, we will identify $F(G_{29})$ and $F(G_{31})$ with the subgroup $E_4$ defined in Proposition \ref{norm}.
Let us denote the derived subgroup of $G_k$ by $G'_k$ for $k \in \{29,31,32\}$. With this notation, we can give a complete classification of
unitary $2$-groups and unitary $3$-groups in the following statement.

\begin{thm}\label{main4} 
Let $V = \C^4$, $\cG = \GL(V)$, and let $G < \cG$ be any finite subgroup. Then the following statements hold.

\begin{enumerate}

\item[\rm (A)] With $E_4$, $\Gamma_4$ and $L$ as defined in Proposition \ref{norm}, we have that
$[\Gamma_4,\Gamma_4]=L = G'_{31}$ and $\Gamma_4 = \bZ(\Gamma_4)G_{31}$. Furthermore, 
$M_{4}(G,V) = M_{4}(\cG,V)$ if and only if one of the following conditions holds

\begin{enumerate}
\item[\rm (A1)] $G = \bZ(G)H$, where $H \cong 2\A_7$ or $H \cong \Sp_4(3) \cong G'_{32}$.
\item[\rm (A2)] $L = [G,G] \leq G < \Gamma_4$. 
\item[\rm (A3)] $E_4 \lhd G < \Gamma_4$, and, after a suitable conjugation in $\Gamma_4$, 
$$G'_{29} = [G,G] \leq G \leq \bZ(\Gamma_4)G_{29}.$$ 
\end{enumerate}

\noindent
In particular, $G < \cH = \U(V)$ is a unitary $2$-group if and only if $G$ is as described in {\rm (A1)--(A3)}.

\item[\rm (B)] $M_6(G,V) = M_6(\cG,V)$ if and only if $G$ is as described in {\rm (A1)--(A2)}.
In particular, $G < \U(V)$ is a unitary $3$-group if and only if $G$ is as described in {\rm (A1)--(A2)}.

\item[\rm (C)] $M_8(G,V) < M_8(\cG,V)$. In particular, no finite subgroup of $\U_4(\C)$ can be a unitary $4$-group. 
\end{enumerate}
 \end{thm}

\begin{proof}
(A) First we assume that $M_4(G,V) = M_4(\cG,V)$, and let $\chi$ denote the character of $G$ afforded by $V$. The same proof as of \cite[Theorem 1.5]{GT} and Theorem \ref{gt-complex} shows that one of the following two possibilities must occur.

\smallskip
$\bullet$ {\bf Almost quasisimple case}:  $S \lhd G/\bZ(G) \leq {\mathrm {Aut}}(S)$ for some finite non-abelian simple group $S$.
By the results of \cite{M}, we have that $S \cong \A_7$ or $\PSp_4(3)$. It is straightforward to check that $E(G) \cong 2\A_7$, 
respectively $\Sp_4(3)$, and furthermore $G$ cannot induce a nontrivial outer automorphism on $S$. Recall that in this case we 
have $F^*(G) = \bZ(G)E(G)$ and so $\bfC_G(E(G)) = \bfC_G(F^*(G)) = \bZ(G)$. It follows that $G = \bZ(G)E(G)$, and (A1) holds. 
Moreover, using \cite{GAP} we can check that $[\alpha^2,\alpha^2] = 2$, $[\alpha^3,\alpha^3] = 6$, but 
$[\alpha^4,\alpha^4] = 38$, respectively $25$, for $\alpha := \chi\da_{E(G)}$. Thus we have checked in the case of (A1) that
$M_{2t}(G,V) = M_{2t}(\cG,V)$ for $t \leq 3$, but $M_8(G,V) > M_8(\cG,V)$, since $M_8(\cG,V) = 24$ by \cite[Lemma 3.2]{GT}.

\smallskip
$\bullet$ {\bf Extraspecial case}: $F^*(G) = F(G) = \bZ(G)E_4$ and $E_4 \lhd G$, in particular, $G \leq \Gamma_4$; furthermore,
$G/\bZ(G)E_4 \leq \Sp(W)$ satisfies conclusion (A)(i) of Theorem \ref{trans} for $W = E_4/\bZ(E_4) \cong \F_2^4$. 
Suppose first that $G/\bZ(G)E_4 \geq \Sp_4(2)' \cong \A_6$. In this case, $G$ induces (at least) all the automorphisms
of $E_4$ that belong to the subgroup $A^{++}$ in the proof of Proposition \ref{norm}. As in that proof, this implies 
that $\bZ(\Gamma_4)G \geq L$. Taking the derived subgroup, we see that 
\begin{equation}\label{for-l}
  [G,G] \geq L, 
\end{equation}  
i.e. we are in the case of (A2). Moreover, 
$$6 = M_6(\cG,V) \leq M_6(G,V) \leq M_6(L,V),$$
and $M_6(L,V) = 6$ as shown above. Hence $M_{2t}(G,V) = M_{2t}(\cG,V)$ for $t \leq 3$.
Applying \eqref{for-l} to $G = G_{31}$ and recalling that $|L| = |G'_{31}|$, we see that $L=G'_{31}$. Next, $G_{31}$ and $\Gamma_4$ induce
the same subgroup $A^+$ of automorphisms of $E_4$, hence $\Gamma_4 = \bZ(\Gamma_4)G_{31}$. Taking the derived subgroup, we obtain
that $L = [\Gamma_4,\Gamma_4]$, and so \eqref{for-l} implies that $[G,G] = L$. 

Next we consider the case where $G/\bZ(G)E_4 = \SL_2(4) \cong \A_5$ or $\SL_2(4) \rtimes C_2 \cong \SSS_5$.
Using \cite{Atlas}, it is easy to check that $\Sp(W) \cong \SSS_6$ has two conjugacy classes $\CL_{1,2}$ of (maximal) subgroups that are isomorphic to 
$\SSS_5$, and two conjugacy classes $\CL'_{1,2}$ of subgroups that are isomorphic to 
$\A_5$. Any member of one class, say $\CL'_1$, is irreducible, but not absolutely irreducible on $W$, that is, 
preserves an $\F_4$-structure on $W$, and is contained in a member of, say $\CL_1$. 
Any member of the other class $\CL_2$ is absolutely irreducible on $W$ 
and preserves a quadratic form $Q$ of type $-$ on $W$; in particular, it 
has two orbits of length $5$ and $10$ on $W \smallsetminus \{0\}$ (corresponding to singular vectors, respectively non-singular vectors,
in $W$ with respect to $Q$), and is contained in a member of $\CL_2$.
On the other hand, since $G$ is transitive on $W \smallsetminus \{0\}$ 
by \cite[Lemma 5.1]{GT}, the last term $G^{(\infty)}$ of the derived series of $G$ must have orbits of only one size on $W \smallsetminus \{0\}$.
Applying this analysis to $K := G_{29}$, we see that $K/E_4$ must belong to $\CL_1$ and the derived subgroup of $K/\bZ(K)E_4$ 
as well as $[K,K]/E_4$ belong to $\CL'_1$. Hence, after a suitable conjugation in $\Gamma_4$, we may assume that 
$$G_{29}/E_4 \geq G/\bZ(G)E_4 \geq G'_{29}/E_4;$$
in particular, the subgroup of automorphisms of $E_4$ induced by $G$ is either the one induced by $G_{29}$, or the one induced by $G'_{29}$.
In either case, we have that
$$G \leq \bZ(\Gamma_4)G_{29},~G'_{29} \leq \bZ(\Gamma_4)[G,G].$$
As $G'_{29}$ is perfect, taking the derived subgroup we obtain that $[G,G] = G'_{29}$, i.e. (A3) holds.
 
\medskip
(B) We have already mentioned above that $M_6(G,V) = M_6(\cG,V)$ for the groups $G$ satisfying (A1) or (A2). By \cite[Lemma 3.1]{GT}, 
it remains to show that for the groups $G$ satisfying (A3), $M_6(G,V) \neq M_6(\cG,V)$. Assume the contrary: $M_6(G,V) = M_6(\cG,V)$. By 
\cite[Remark 2.3]{GT}, this equality implies that $G$ is irreducible on all the simple $\cG$-submodules of $V \otimes V \otimes V^*$,
which can be seen using \cite[Appendix A.7]{Lu} to decompose as the direct sum of simple summands of dimension $4$ (with multiplicity $2$), $20$, and $36$. Let $\theta$ denote 
the character of $G$ afforded by the simple $\cG$-summand of dimension $36$. Note that $\chi$ vanishes on 
$F(G) \smallsetminus \bZ(G)$ and faithful on $\bZ(G)$. It follows that 
$$\chi^2\overline\chi\da_{F(G)} = 16\chi\da_{F(G)}.$$ 
As $\chi\da_{F(G)}$ is irreducible, we see that $\theta\da_{F(G)} = 9(\chi\da_{F(G)})$. 
But $\chi\da_{F(G)}$ obviously extends to $G \rhd F(G)$. It follows by Gallagher's theorem \cite[(6.17)]{Is} that 
$G/F(G)$ admits an irreducible character $\beta$ of degree $9$ (such that $\theta\da_G = (\chi\da_G)\beta$). 
This is a contradiction, since $G/F(G) \cong \A_5$ or $\SSS_5$.

\medskip
(C) Assume the contrary: $M_8(G,V) = M_8(\cG,V)$. 
Then $M_6(G,V) = M_6(\cG,V)$ by \cite[Lemma 3.1]{GT}. By (B), we may assume that $G$ satisfies (A1) or (A2). 
By \cite[Remark 2.3]{GT}, the equality $M_8(G,V) = M_8(\cG,V)$ implies that $G$ is irreducible on the simple 
$\cG$-submodule ${\mathrm {Sym}}^4(V)$ (of dimension $35$) of $V^{\otimes 4}$. This in turn implies, for instance by Ito's theorem 
\cite[(6.15)]{Is} that $35$ divides $|G/\bZ(G)|$. The latter condition rules out (A2) since $|G/\bZ(G)|$ divides $2^4 \cdot |\Sp_4(2)|$ in 
that case. Finally, we already mentioned above that $M_8(G,V) > M_8(\cG,V)$ in the case of (A1).    
\end{proof}

To handle the remaining cases $d = 2,3$, we first note:

\begin{lem}\label{d2}
Let $\cG = \SL(V)$ for $V = \C^2$. Then the following statements hold.
\begin{enumerate}[\rm(i)]
\item $M_6(\cG,V) = 5$, $M_8(\cG,V) = 14$, and $M_{10}(\cG,V) = 42$.
\item Suppose $M_{2t}(G,V) = M_{2t}(\cG,V)$ for a finite group $G < \cG$. If $t \geq 4$ then $5$ divides $|G/\bZ(G)|$. If $t \geq 6$ then
$7$ divides $|G/\bZ(G)|$.
\item Suppose $\SL_2(5) \cong G < \cG$. Then $M_{2t}(G,V) = M_{2t}(\cG,V)$ for $1 \leq t \leq 5$ but 
$M_{2t}(G,V) > M_{2t}(\cG,V)$ for $t \geq 6$.
\end{enumerate}
\end{lem}

\begin{proof}
Note that the symmetric powers $\Sym^k(V)$, $k \geq 0$, are pairwise non-isomorphic irreducible $\C\cG$-modules, with 
$\Sym^0(V) \cong \C \cong \wedge^2(V)$, and $V \otimes V \cong \Sym^2(V) \oplus \C$. Now using 
\cite[Exercise 11.11]{FH} we obtain for all $a \geq 1$ that 
$$\Sym^a(V) \oplus V \cong \Sym^{a+1}(V) \oplus \Sym^{a-1}(V)$$
as $\C\cG$-modules. It follows that     
$$\begin{array}{l}
    V^{\otimes 3} \cong \Sym^3(V) \oplus V^{\oplus 2},\\
    V^{\otimes 4} \cong \Sym^4(V) \oplus (\Sym^2(V))^{\oplus 3} \oplus \C^{\oplus 2},\\
    V^{\otimes 5} \cong \Sym^5(V) \oplus (\Sym^3(V))^{\oplus 4}  \oplus V^{\oplus 5} \end{array}$$
as $\C\cG$-modules (with the superscripts indicating the multiplicities), implying (i).

For (ii), note by Remark 2.3 and Lemma 3.1 of \cite{GT} that the assumption implies that  $G$ is irreducible on $\Sym^4(V)$ of dimension $5$ if $t \geq 4$, and on $\Sym^6(V)$ of dimension $7$ if $t \geq 6$.

The first assertion in (iii) can be checked using (i) and \cite{GAP}, and the second assertion follows from (ii).
\end{proof}

Now we recall three complex reflection groups $G_4 \cong \SL_2(3)$, $G_{12} \cong \GL_2(3)$, and $G_{16} \cong C_5 \times \SL_2(5)$ in dimension $d=2$, listed on lines
4, 12, and 16 of \cite[Table VII]{ST}, and three complex reflection groups $G_{24} \cong C_2 \times \SL_3(2)$,  
$G_{25} \cong 3^{1+2}_+ \rtimes \SL_2(3)$, and $G_{27} \cong C_2 \times 3\A_6$ in dimension $d=3$, listed on lines
24, 25, and 27 of \cite[Table VII]{ST}. 
As above, for any of these $6$ groups $G_k$, $G'_k$ denotes its derived subgroup.
A direct calculation using the computer packages {{\sf {GAP}}3} \cite{Mi}, \cite{S+}, and {\sf {Chevie}}
\cite{GHMP}, shows that each of these $6$ groups $G$, being embedded in $\cH = \U_d(\C)$, is a unitary $2$-group; furthermore,
$G_{12}$, $G'_{16}$, and $G'_{27}$ are unitary $3$-groups. One can check that
$F(G_4) \cong F(G_{12})$ is a quaternion group $Q_8 = 2^{1+2}_-$, and we will identify them with an irreducible subgroup
$E_2 \cong Q_8$ of $\GL_2(\C)$. Also, $E_3 := F(G_{25}) \cong 3^{1+2}_+$ is an extraspecial $3$-group of order $27$ and exponent $3$,
which is an irreducible subgroup of $\GL_3(\C)$. 
Let $\Gamma_d := \bfN_{\GL_d(\C)}(E_d)$ for $d = 2,3$. Now we can give a complete classification of unitary $t$-groups in 
dimensions $2$ and $3$.

\begin{thm}\label{main23}
Let $V = \C^d$ with $d = 2$ or $3$, $\cG = \GL(V)$, and let $G < \cG$ be any finite subgroup. Then the following statements hold.

\begin{enumerate}

\item[\rm (A)] Suppose $d=2$. Then 
$M_{4}(G,V) = M_{4}(\cG,V)$ if and only if one of the following conditions holds

\begin{enumerate}
\item[\rm (A1)] $G = \bZ(G)H$, where $H = G'_{16} \cong \SL_2(5)$.
\item[\rm (A2)] $E_2 \lhd G < \Gamma_2$ and $\bZ(\cG)G = \bZ(\cG)H$, where $H = G_{12} \cong \GL_2(3)$.
\item[\rm (A3)] $E_2 \lhd G < \Gamma_2$ and $\bZ(\cG)G = \bZ(\cG)H$, where $H = G_4 \cong \SL_2(3)$.
\end{enumerate}

\noindent
In particular, $G < \cH = \U(V)$ is a unitary $2$-group if and only if $G$ is as described in {\rm (A1)--(A3)}.
Furthermore, $G < \cH = \U(V)$ is a unitary $3$-group if and only if $G$ is as described in {\rm (A1)--(A2)}.
Moreover, such a subgroup $G$ can be a unitary $t$-group for some $t \geq 4$ if and only if $4 \leq t \leq 5$ and 
$G$ is as described in {\rm (A1)}.

\item[\rm (B)] Suppose $d=3$. Then 
$M_{4}(G,V) = M_{4}(\cG,V)$ if and only if one of the following conditions holds

\begin{enumerate}
\item[\rm (B1)] $G = \bZ(G)H$, where $H = G'_{27}\cong 3\A_6$.
\item[\rm (B2)] $G = \bZ(G)H$, where $H = G'_{24} \cong \SL_3(2)$.
\item[\rm (B3)] $E_3 \lhd G < \Gamma_3$. Moreover, either $\bZ(\cG)G = \bZ(\cG)G'_{25}$, or 
$\bZ(\cG)G = \bZ(\cG)G_{25}$. 
\end{enumerate}
\noindent
In particular, $G < \cH = \U(V)$ is a unitary $3$-group if and only if $G$ is as described in {\rm (B1)}, and 
no finite subgroup of $\U(V)$ can be a unitary $4$-group. 
\end{enumerate}

\end{thm}

\begin{proof}
Let $G < \cG$ be any finite subgroup such that $M_{2t}(G,V) = M_{2t}(\cG,V)$ for some $t \geq 2$;
in particular, 
\begin{equation}\label{d231}
  M_4(G,V) = M_4(\cG,V).
\end{equation}    
First we note that if $K < \cG$ is any finite subgroup
that is {\it equal to $G$ up to scalars}, i.e. $\bZ(\cG)G = \bZ(\cG)K$, then by \cite[Remark 2.3]{GT} we see that 
$M_{2t}(K,V) = M_{2t}(\cG,V)$. So, instead of working with $G$, we will work with the following finite subgroup
$$K := \{ \lambda g \mid g \in G,\lambda \in \C^{\times}, \det(\lambda g) = 1 \} < \SL(V).$$
Next, we observe that $G$ acts primitively on $V$. (Otherwise $G$ contains a normal abelian subgroup $A$ with 
$G/A \hookrightarrow \SSS_d$. In this case, by Ito's theorem $G$ cannot act irreducibly on the irreducible $\cG$-submodule of
dimension $d^2-1$ of $V \otimes V^*$, and so $G$ violates \eqref{d231} by \cite[Remark 2.3]{GT}.) Now, using the 
fact that $d=\dim(V) \leq 3$ is a prime number, it is straightforward to show that one of the following two possibilities must occur.

\smallskip
$\bullet$ {\bf Almost quasisimple case}: $S \lhd G/\bZ(G) \leq {\mathrm {Aut}}(S)$ for some finite non-abelian simple group $S$.
By the results of \cite{M}, we have that $S \cong \PSL_2(5)$ if $d=2$, and $S \cong \SL_3(2)$ or $\A_6$ if $d=3$. Arguing as 
in the proof of Theorem \ref{main4}, we see that (A1), (B1), or (B2) holds. In the case of (A1), $M_{2t}(G,V) = M_{2t}(\cG,V)$ if 
and only if $2 \leq t \leq 5$ by Lemma \ref{d2}. In the case of (B2), $G$ cannot act irreducibly on $\Sym^3(V)$ of dimension 
$10$, whence $M_{2t}(G,V) = M_{2t}(\cG,V)$ if and only if $t=2$. Assume we are in the case of (B1). As mentioned above,
then we have $M_{2t}(G,V) = M_{2t}(\cG,V)$ for $t = 2,3$. However, if $\varpi_1$ and $\varpi_2$ denote the two fundamental weights
of $[\cG,\cG] \cong \SL_3(\C)$, then $V^{\otimes 2} \otimes (V^*)^{\otimes 2}$ contains an irreducible $[\cG,\cG]$-submodule with
highest weight $2\varpi_1+2\varpi_2$ of dimension $27$ (see \cite[Appendix A.6]{Lu}). Clearly, $G$ cannot act irreducibly on 
this submodule, and so $M_8(G,V) > M_8(\cG,V)$ by \cite[Remark 2.3]{GT}.

\smallskip
$\bullet$ {\bf Extraspecial case}: $F^*(G) = F(G) = \bZ(G)E_d$ and $E_d \lhd G$, in particular, $G \leq \Gamma_d$; furthermore,
$G/\bZ(G)E_d \leq \Sp(W)$ satisfies conclusion (A)(i) of Theorem \ref{trans} for $W = E_d/\bZ(E_d) \cong \F_d^2$. The latter condition
is equivalent to require $G/\bZ(G)E_d$ to contain the unique subgroup $C_3$ of $\Sp_2(2) \cong \SSS_3$ when $d=2$ and the unique subgroup
$Q_8$ of $\Sp_2(3) \cong \SL_2(3)$ when $d=3$. Note that $G_4 \cong \SL_2(3)$, respectively $G_{12} \cong \GL_2(3)$, 
induces the subgroup $C_3$, respectively $\SSS_3$, of outer automorphisms of $E_2 \cong Q_8$. Similarly, 
$G'_{25} \cong 3^{1+2}_{+} \rtimes Q_8$, respectively $G_{25} \cong 3^{1+2}_{+} \rtimes \SL_2(3)$, 
induces the subgroup $Q_8$, respectively $\SL_2(3)$, of outer automorphisms of $E_3 \cong 3^{1+2}_{+}$ that act trivially on 
$\bZ(E_3)$. Now arguing as in the proof of Theorem \ref{main4}, we see that (A2), (A3), or (B3) holds. In the case of (A3),
$M_8(G,V) > M_8(\cG,V)$ by Lemma \ref{d2}, and we already mentioned above that $M_6(G,V) = M_6(\cG,V)$.
In the case of (A2), $G$ cannot act irreducibly on $\Sym^3(V)$ of dimension $4$, so $M_{2t}(G,V) = M_{2t}(\cG,V)$ if and only if $t=2$. 
In the case of (B3), $G$ cannot act irreducibly on $\Sym^3(V)$ of dimension $10$, so $M_{2t}(G,V) = M_{2t}(\cG,V)$ if and only if $t=2$. 
\end{proof}

\end{document}